\documentclass[12pt,twoside]{article}

\usepackage{a4}

\usepackage{amssymb,amsmath,amsthm,latexsym}
\usepackage{amsfonts}
\usepackage{amsfonts}
\usepackage{graphicx}

\numberwithin{subcase}{case}

\usepackage{amsmath, amsfonts}
\usepackage{amssymb, graphicx}
\usepackage{amscd}
\usepackage{textcomp}
\usepackage{palatino}
\usepackage{xcolor}
\usepackage{array}
\usepackage[colorlinks=true,linkcolor=red,citecolor=red]{hyperref}
\allowdisplaybreaks

\newtheorem{theorem}{Theorem}[section]

\newtheorem{corollary}[theorem] {Corollary}

\newtheorem{example}[theorem]{Example}
\newtheorem{lemma} [theorem]{Lemma}

\newtheorem{proposition}[theorem]{Proposition}

\newcommand{\Z}{{\mathbb Z}}

\newcommand{\Q}{{\mathbb Q}}

\newcommand{\N}{{\mathbb N}}

\newcommand{\F}{{\mathbb F}}

\newtheoremstyle{case}{}{}{}{}{}{:}{ }{}
\theoremstyle{case}

\usepackage{indentfirst}

\vspace{5cm}

\begin{document}
	\label{'ubf'}  
	\setcounter{page}{1} %Put here the starting page number
	\markboth
	{\hspace*{-9mm} \centerline{\footnotesize
			% Put here the left page top label 
			Characterization of prime divisors}}
	{\centerline{\footnotesize 
			%%put here the author's name
			Characterization of prime divisors
		} \hspace*{-9mm}}
	
	\vspace*{-2cm}

	\begin{center}
		
		{\textbf{On characterization of prime divisors of the index of a quadrinomial}\\
			\vspace{.2cm}
			\medskip
			{\sc Tapas Chatterjee}\\
			{\footnotesize  Department of Mathematics,}\\
			{\footnotesize Indian Institute of Technology Ropar, Punjab, India.}\\
			{\footnotesize e-mail: {\it tapasc@iitrpr.ac.in}}
			
			\medskip
			{\sc Karishan Kumar}\\
			{\footnotesize Department of Mathematics, }\\
			{\footnotesize Indian Institute of Technology Ropar, Punjab, India.}\\
			{\footnotesize e-mail: {\it karishan.22maz0012@iitrpr.ac.in}}
			\medskip}
		
	\end{center}
	\thispagestyle{empty} 
	\vspace{-.4cm}
	
	\hrulefill
	\begin{abstract}  
		{\footnotesize }
		Let $\theta$ be an algebraic integer and $f(x)=x^{n}+ax^{n-1}+bx+c$ be the minimal polynomial of $\theta$ over the rationals. Let $K=\Q(\theta)$ be a number field and $\mathcal{O}_{K}$ be the ring of integers of $K.$ In this article, we characterize all the prime divisors of the discriminant of $f(x)$ which do not divide the index of $f(x).$ As a fascinating corollary, we deduce necessary and sufficient conditions for $\Z[\theta]$ to be integrally closed, where $\theta$ is associated with certain quadrinomials.
		
	\end{abstract}
	\hrulefill
	
	\noindent 
	{\small \textbf{Key words and phrases}: Dedekind criterion; Discriminant; Index of an algebraic integer; Monogenic number fields; Ring of algebraic integers.\\
		
		\noindent
		{\bf{Mathematics Subject Classification 2020:}} Primary: 11R04, 11R29, 11Y40; Secondary: 11R09, 11R21.
		
		\vspace{-.37cm}
		
		\section{\bf Introduction}
		
		Let $\mathcal{O}_{K}$ be the ring of algebraic integers of a number field $K=\mathbb{Q}(\theta),$ where $\theta$ is
		an algebraic integer with minimal polynomial $f(x)$ over the field $\mathbb{Q}.$
		In $1878,$ Dedekind proved a remarkable theorem that established a connection between the decompositions of the polynomial $f(x)$ modulo $p$ and the factorization of $p\mathcal{O}_{K}$ into a product of prime ideals of $\mathcal{O}_{K}.$ He proved the following theorem \cite[Theorem 4.33]{WN}:
		\begin{theorem}
			Let $p$ be a prime and $f(x)$ be the minimal polynomial of an algebraic integer $\theta$ over the field $\mathbb{Q}$ such that
			$$\bar{f}(x)=\bar {f}_1(x)^{a_1}\bar {f}_2(x)^{a_2}\cdots\bar {f}_t(x)^{a_t}$$ be the factorization of $\bar{f}(x)$ as a product of powers of distinct monic irreducible polynomials over the field $\F_{p},$ where each $f_i(x)\in \mathbb{Z}[x]$ and $p \nmid [\mathcal{O}_{K}:\mathbb{Z}[\theta]].$ Then,  $$p\mathcal{O}_{K}= p_{1}^{a_1}.p_{2}^{a_2}\cdots p_{t}^{a_t},$$ where $p_{1},$ $p_{2},$ $\ldots p_{t}$ are distinct prime ideals of $\mathcal{O}_{K},$ $p_{j}=p\mathcal{O}_{K}+{f}_{j}(\theta)\mathcal{O}_{K},$ and the norm of these prime ideals $p_{j}$ is equal to $p^{deg({f}_j(x))},$ for all $j\in\{1, 2,\ldots, t\}.$
		\end{theorem}
		But the converse of the theorem was proved after a long time in $2008$ \cite {skm}. Further, Dedekind proved a useful criterion known as Dedekind criterion (\ref{T4}) which plays a very important role in finding out the prime factors of the index of $f(x).$ Recently, many mathematicians \cite{CK2, CK3, CK4, CK5, jss, skm} have proved some results related to monogenity of algebraic number fields associated to trinomials and a specific category of quadrinomials. In this direction, we investigate the case of quadrinomials.
		In this article, we use the Dedekind criterion to characterize the primes which divide $[\mathcal{O}_{K}:\mathbb{Z}[\theta]],$ where $[\mathcal{O}_{K}:\mathbb{Z}[\theta]]$ represents the index of $\mathbb{Z}[\theta]$ in $\mathcal{O}_{K}$ and $\theta$ is a root of the irreducible polynomial $$f(x)=x^{n}+ax^{n-1}+bx+c\in \mathbb{Z}[x],$$ with $n>4, ~abc\neq 0,$ $n^{2}=ak,$ $k \in \mathbb{N}.$ This gives us necessary and sufficient conditions for $\Z[\theta]$ to be integrally closed, which depends only on $a,~ b,~ c,$ and $n.$ Alternatively, we verify whether the set $\{1, \theta, \theta^{2}, \theta^{3}, \ldots, \theta^{n-1}\}$ is an integral basis  of $K$ or not.
		
		Also, we use the widely recognized formula of Dedekind $D(f)= [\mathcal{O}_{K}:\mathbb{Z}[\theta]]^{2}D_{K}$ to verify the monogenity of the corresponding number field, where $D(f)$ stands for the discriminant of $f(x)$ and $D_{K}$ stands for the discriminant of the number field $K.$ Finally, we provide some examples which state the importance of the given theorems. Throughout the paper, $\bar F(x)$ denotes the operation reduction modulo $p$ for any polynomial $F(x)$ and $D$ denotes the discriminant of the polynomial $$f(x)= x^{n}+ax^{n-1}+bx+c,$$ where $n>4,$ $abc\neq 0.$ The discriminant of $f(x)$ is defined as \cite[Theorem 4.1]{KG}
		\begin{equation}\label{E1}
			\begin{split}
				D&=(-1)^{\frac{(n+2)(n-1)}{2}}\bigg[(n-1)^{n-1}a^{n}c^{n-2}+\frac{n^{2}(n-1)^{n-1}b^{n-1}c}{a}+\frac{(n-1)^{n-3}b^{n-2}(n^{2}c-ab)^{2}}{a^{2}}\\&-\frac{n(n-1)^{n-3}b^{n-2}(n^{2}c-ab)((n-2)ab+cn)}{a^{2}}\\&+\sum_{i=0}^{\lfloor \frac{n-3}{2} \rfloor}\bigg[\frac{2n(n-1)^{2}abc[(2-n)ab-cn]^{n-3-2i}[((n-2)ab+cn)^{2}-4abc(n-1)^{2}]^{i}\binom{n-3}{2i}}{2^{n-3}}\\&+\frac{(n^{2}c-ab)[(2-n)ab-cn]^{n-2-2i}[((n-2)ab+cn)^{2}-4abc(n-1)^{2}]^{i}\binom {n-3}{2i}(n-2)}{2^{n-3}(n-2-2i)}\bigg]\bigg]\\&-(1+(-1)^{n})\frac{(n^{2}c-ab)[((n-2)ab+cn)^{2}-4abc(n-1)^{2}]^\frac{n-2}{2}}{2^{n-2}}.
			\end{split}
		\end{equation}
		Furthermore, we define $ M(x)$ such that 
		\begin{equation}\label {E1}
			M(x)=\frac{1}{p}\bigg[f(x)-\displaystyle\prod_{i=1}^{r}  G_{i}(x)^{e_{i}}\bigg],
		\end{equation}
		where $G_{i}(x)$ are monic lifts of $\bar G_{i}(x)$ given $\bar f(x)=\bar G_{1}(x)^{e_{1}}\bar G_{2}(x)^{e_{2}}\ldots\bar G_{r}(x)^{e_{r}},$ and $e_{i}\in \N,$ for all positive integers $i\in \{1, 2,\ldots, r\}.$
		In this regard, we have the following theorem.
		\begin{theorem}\label{th1}
			Let $\theta$ be an algebraic integer, $n>4$ be a positive integer, and $$f(x)= x^{n}+ax^{n-1}+bx+c$$ be the minimal polynomial of $\theta$ over $\mathbb{Q},$ where $abc\neq0,$ $n^{2}=ak,$ $\gcd(a,k)=1,$ $k \in \mathbb{N}.$ Let $K=\mathbb{Q}(\theta)$ and $\mathcal{O}_{K}$ be the ring of algebraic integers of $K.$ A prime factor $p$ of the discriminant $D$ of $f(x)$ does not divide $[\mathcal{O}_{K}:\mathbb{Z}[\theta]]$ if and only if $p$ satisfies one of the following conditions:  
			\begin{enumerate}
				\item When $p|a,$ $p|b$ and $p|c,$ then $p^{2}\nmid c.$
				\item When $p | a,$ $p | b,$ $p\nmid c,$ then either\\
				(i) $p^{2}|b$ and $p\nmid c_{1}$\\
				or\\
				(ii) $p\nmid b_{1}[(-c_{1})^{n}+c(b_{1})^{n}],$ where $b=pb_{1},$ $c_{1}=\frac{(c+(-c)^{p^{r}})}{p},$ and $p^{r}||n.$
				
				\item When $p\nmid a,$ $p|b$ and $p|c,$ then $p^{2}\nmid c$ and $p^{2}\nmid (ab-c).$
				%				\item $p|a,$ $p\nmid b$ and $p| c.$
				%				
				%				\item	 $p|a$ and $p\nmid bc.$ 
				\item When $p\nmid ab$ and $p|c,$ 
				then one of the following conditions is satisfied:\\
				(i) $p|(n-1)$\\
				(ii) If $p\nmid(n-1),$ then either $a(-a(n-2))^{n-2}+b(n-1)^{n-1}\not\equiv 0\pmod{p}$\\ or\\ if $ a(-a(n-2))^{n-2}+b(n-1)^{n-1}\equiv 0\pmod{p},$ then $(x-\bar n_{1})\nmid \bar M(x),$ where 
				$\bar n_{1}=-\overline{(n-1)}^{-1}\bar a~ \overline{(n-2)}.$ 
				\item When $p\nmid ac$ and $p|b,$ then one of the following conditions is satisfied:\\
				(i) $p|n$ \\
				(ii) If $p\nmid n,$ then either $cn^{n}+a(-a(n-1))^{n-1}\not\equiv 0\pmod{p}$\\ or\\ if $cn^{n}+a(-a(n-1))^{n-1}\equiv 0\pmod{p},$ then $(x-\bar n_{2})\nmid \bar M(x),$ where $\bar n_{2}=-(\overline{n})^{-1}\bar a~ \overline{(n-1)}.$ 
				
				\item When $p\nmid abc,$ then one of the following conditions is satisfied:\\
				(I) If $p$ is an odd prime, then one of the following conditions is satisfied:
				
				(i) If $p|(n-1),$ then either $ab\not\equiv c\pmod{p}$ or if $ab\equiv c\pmod{p},$ then either $p|v_{1}$ and $p\nmid v_{0}$ or $p\nmid v_{1}[(- v_{0})^{n}+ {a} v_{1}(- v_{0})^{n-1}- b ( v_{1})^{n-1} v_{0}+ {c}( v_{1})^{n}],$ where $v_{1}=\frac{(b+(-b)^{p^{r_0}})}{p}$ and $v_{0}=\frac{(c+a(-b)^{p^{r}})}{p}.$ \\
				(ii) If $p|n,$ then either $$(-\bar a\pm [(\bar a)^{2}-\bar c\bar a(\bar b)^{-1}]^{\frac {1}{2}})^{n-2}\bar a \neq \bar b$$ or $x^{2}+2\bar ax+\bar a\bar c(\bar b)^{-1}$ is co-prime to $\bar M(x).$ 
				
				(iii) If $p\nmid n(n-1),$ then either $$(l_{1})^{n-2}(\bar {n}l_{1}+2\bar a ~\overline {(n-1)})+2^{n-1}\bar {b}\neq \bar 0$$ with 
				\begin{equation*}
					\begin{split}
						l_{1}&=\{-(\bar b-\bar n\bar b)^{-1}(\bar a\bar b-\overline{(n-1)}\bar a \bar b-\bar c\bar n)\\&\pm[((\bar b-\bar n\bar b)^{-1}(\bar a\bar b-\overline{(n-1)}\bar a \bar b-\bar c\bar n))^{2}+4(\bar b-\bar n\bar b)^{-1}\overline{(n-1)}\bar a \bar c]^\frac{1}{2}\}
					\end{split}
				\end{equation*}
				
				or $$x^{2}+x(\bar b-\bar n\bar b)^{-1}(\bar a\bar b-\overline{(n-1)}\bar a \bar b-\bar c\bar n)-(\bar b-\bar n\bar b)^{-1}\overline{(n-1)}\bar a \bar c$$ is co-prime to $\bar M(x).$

				(II) If $p$ is an even prime, then one of the following conditions is satisfied:\\
				(i) If $p=2$ and $n$ is even, then $(x+1)\nmid \bar M(x).$ 
				
				(ii) If $p=2,$ $n$ is odd then 
				exactly one of the elements in the set $\{ \frac{b+1}{2},~  \frac{a+c}{2}\}$ is divisible by two.

			\end{enumerate}
		\end{theorem}
		
		As a consequence of the theorem, we have the following important corollary. 
		\begin{corollary}\label{C1}
			Let $K=\mathbb{Q}(\theta)$ be a number field corresponding to the minimal polynomial $f(x)= x^{n}+ax^{n-1}+bx+c\in\mathbb{Z}[x]$ of $\theta.$ Then, $\mathcal{O}_K=\Z[\theta]$ if and only if for each prime $p$ dividing the discriminant $D$ of $f(x),$ satisfies one of the conditions (1) to (8) of the Theorem (\ref{th1}).
			
		\end{corollary}

		A generalization of Theorem (\ref {th1}) has been established in the further upcoming paper \cite {CK2}.
		Here, we recure the following proposition which characterizes some prime factors $p$ of $D_{K}.$    
		\begin{proposition} \label{P1}
			Let $\theta$ be an algebraic integer, $n>4$ be a positive integer,  and $$f(x)= x^{n}+ax^{n-1}+bx+c$$ be the minimal polynomial of $\theta$ over $\mathbb{Q},$ where $n^{2}=ak,$ $k \in \mathbb{N},$ and $abc\neq 0.$ Let $K=\mathbb{Q}(\theta)$ and $\mathcal{O}_{K}$ be the ring of algebraic integers of $K.$
			Let $p$ be an odd prime satisfying the following conditions,
			$p \nmid a,$ $p \nmid b,$ $p| c$ and $p|(n-1).$ Then, $p |D_{K}$ if and only if $$
			\left\{\begin{array}{ll}
				\displaystyle\sum_{i=0}^{\lfloor \frac{n-3}{2} \rfloor}\binom{n-3}{2i}\left[\frac{n-2}{n-2-2i}\right] \equiv 1\pmod{p}, & {\rm if }\; n \text{~is~even}\\
				\displaystyle\sum_{i=0}^{\lfloor \frac{n-3}{2} \rfloor}\binom{n-3}{2i}\left[\frac{n-2}{n-2-2i}\right]\equiv 0\pmod{p}, & {\rm if}\;n~\text{is~ odd.} 
			\end{array}\right.$$ 
		\end{proposition}
		From the structure of $D,$ it is difficult to factorize it for large coefficients of associated polynomials $f(x).$ But the following proposition provides us certain conditions by which we can generate a list of primes $p$ such that $p\nmid D$ as well as $p\nmid [\mathcal{O}_{K}:\mathbb{Z}[\theta]],$ for $n^{2}=ak.$ 
		\begin{proposition} \label{T2}
			Let $\theta$ be an algebraic integer and $$f(x)= x^{n}+ax^{n-1}+bx+c$$ be the minimal polynomial of $\theta$ over $\mathbb{Q}$ with $n^{2}=ak,$ $k \in \mathbb{N},$ $n>4,$ and $abc\neq 0.$ Let $K=\mathbb{Q}(\theta)$ and $\mathcal{O}_{K}$ be the ring of algebraic integers of $K.$ An odd prime $p$ does not divide $[\mathcal{O}_{K}:\mathbb{Z}[\theta]]$ as well as $D,$ if it satisfies one of the following conditions:
			\begin{enumerate}
				\item  $p\nmid a,$ $p\nmid c ,$ $p|b$ and $p|n.$ 
				\item  $p\nmid a,$ $p\nmid b,$ $p|c$ and $p|(n-2).$
			\end{enumerate}
		\end{proposition}

		\section{\bf Notations and Preliminaries}
		In this section, we define some basic notations and preliminaries.
		Let $p$ be any prime. For any integer $m$ such that $p\nmid m,$ then $(\bar m)^{-1}$ denotes the inverse of $\bar m$ in the finite field $\F_{p}=\mathbb{Z}/p\mathbb{Z}.$
		We need the following results that play a crucial role in the proof of Theorem (\ref {th1}). 
		\begin{lemma}\label{L1}
			Let $n,$ $a,$ $k\in \mathbb{N},$ where $n^{2}=ak$ and $\gcd(a, k)=1.$ If a prime $p | a,$ then $p^{2}|a.$
		\end{lemma}
		\begin{proof}
			Let $p|a,$ then $p|n^{2}.$ Since, $p$ is a prime number implies that $p$ divides $n,$ it follows that $p^{2} | n^{2}$ i.e. $p^{2}$ divides $ak.$ Thus, the coprimality of $a$ and $k$ gives us $p^{2} | a.$
		\end{proof}
		
		In 1878, Dedekind introduced the following criterion known as the Dedekind criterion  (\cite[Theorem 6.1.4]{HC}, \cite{RD}), which provides necessary and sufficient conditions to be satisfied by $f(x)$ so that a prime number $p$ does not divide the index $[\mathcal{O}_{K}:\mathbb{Z}[\theta]].$
		\begin{theorem}\label{T4}
			(Dedekind Criterion) Let $\theta$ be an algebraic integer and $f(x)$ be the minimal polynomial of $\theta$ over $\mathbb{Q}.$ Let $K=\mathbb{Q}(\theta)$ be the corresponding number field. Let $p$ be a prime and $$\bar{f}(x)=\bar {f}_1(x)^{a_1}\bar {f}_2(x)^{a_2}\cdots\bar {f}_t(x)^{a_t}$$ be the factorization of $\bar{f}(x)$ as a product of powers of distinct monic irreducible polynomials over the field $\F_{p}.$ Let $M(x)$ be the polynomial defined as $$M(x) = \frac{1}{p}(f(x)-{f}_1(x)^{a_1} {f}_2(x)^{a_2} \cdots {f}_t(x)^{a_t})\in\mathbb{Z}[x],$$  where $f_i(x)\in \mathbb{Z}[x]$ are monic lifts of $\bar f_i(x),$ for all $i=1,~2,\cdots,t.$ Then, $p \nmid [\mathcal{O}_{K}:\mathbb{Z}[\theta]]$ if and only if for each $i,$ we have either $a_i=1$ or $\bar{f_i}(x)$ does not divide $\bar{M}(x).$  
		\end{theorem}
		
		The following lemma is helpful to prove our main result.
		
		\begin{lemma}{\rm \cite{jss}}\label{L3}
			Let $A(x)=x^{l'}+ax^{l}+b\in\mathbb{Z}[x]$ be a polynomial of degree $l'$. Let $p$ be a prime number and $k$ be any natural number. Then, there exist a polynomial $U(x)\in\mathbb{Z}[x]$ such that $$A(x^{p^{k}})=A(x)^{p^{k}}+p*A(x)U(x)+ax^{lp^{k}}+b+(-ax^{l}-b)^{p^{k}}.$$
		\end{lemma}
		
		Now, we present two lemmas which generalize the first and third part of the Theorem (\ref{th1}). 
		\begin{lemma}\label{L5}
			Let $\theta$ be an algebraic integer, $n\geq 2$ be any integer, and $$f(x)= x^{n}+ a_{n-1}x^{n-1}+ a_{n-2}x^{n-2}+\cdots+a_{1}x+ a_{0},$$ be the minimal polynomial of $\theta$ over $\mathbb{Q}.$ Let $K=\mathbb{Q}(\theta)$ be the corresponding number field. Let $p$ be a prime number which divides $a_{i},$ for all $i= 0, 1, 2,\ldots,(n-1).$ Then, $p\nmid [\mathcal{O}_{K}:\mathbb{Z}[\theta]]$ if and only if $p^{2}\nmid  a_{0}.$ 
		\end{lemma}
		\begin{proof}
			Let $p$ be any prime number such that $p|a_{i},$ for all $i= 0, 1,2, \ldots,(n-1).$ Then, $$f(x)= x^{n}+ a_{n-1}x^{n-1}+ a_{n-2}x^{n-2}+\cdots+a_{1}x+ a_{0} \equiv x^{n} \pmod{p}$$ which implies that $\bar {f}(x)=x^{n}\in \F_{p}[x].$ As $n\geq 2,$ by Dedekind criterion, $p\nmid [\mathcal{O}_{K}:\mathbb{Z}[\theta]]$ if and only if $x$ does not divide $\bar {M}(x),$ where $$M(x)=\frac{a_{n-1}x^{n-1}+ a_{n-2}x^{n-2}+\cdots+a_{1}x+ a_{0}}{p}.$$ Here, $x$ divides $\bar {M}(x)$ if and only if $p^{2}|a_{0}$ or in other words $x$ does not divide $\bar {M}(x)$ if and only if $p^{2}\nmid a_{0}.$ Thus, $p\nmid [\mathcal{O}_{K}:\mathbb{Z}[\theta]]$ if and only if $p^{2}\nmid  a_{0}.$ This completes the proof.   
		\end{proof}
		
		\begin{lemma}\label{L6}
			Let $\theta$ be an algebraic integer, $n \geq 3$ be any integer, and $$f(x)= x^{n}+ax^{n-1}+bx+c$$ be the minimal polynomial of $\theta$ over $\mathbb{Q}.$ Let $K=\mathbb{Q}(\theta)$ and $\mathcal{O}_{K}$ be the ring of algebraic integers of $K.$ Let $p$ be a prime number such that $p\nmid a,$ $p|b,$ and $p|c.$ Then, $p | [\mathcal{O}_{K}:\mathbb{Z}[\theta]]$ if and only if either $p^{2}|c$ or $\bar 
			{ \frac{c}{p}} =\bar{a}\bar {\frac{b}{p}},$ where $\bar {a},$ $\bar {\frac{b}{p}},$ $\bar {\frac{c}{p}}$ $\in \F_{p}.$
		\end{lemma}
		\begin{proof}
			Let $p$ be a prime number satisfying $p\nmid a,$ $p|b,$ and $p|c.$ Then, $$f(x)=x^{n}+ax^{n-1}+bx+c\equiv x^{n}+ax^{n-1} \pmod{p},$$ i.e. $$\bar {f}(x)=x^{n}+\bar{a}x^{n-1}= x^{n-1}(x+\bar{a}).$$ Here, monic distinct irreducible factors of $\bar{f}(x)$ are $\bar{g_1}(x)=x$ and $\bar{g_2}(x)=x+\bar{a}.$ Then, by Dedekind criterion, we see that prime $p | [\mathcal{O}_{K}:\mathbb{Z}[\theta]]$ if and only if for some $i\in\{1, 2\},$ $\bar{g_i}(x)$ divides $\bar{M}(x)$ ( because, $n-1\geq 2$), where
			\begin{align*}
				M(x)&=\frac{1}{p}[f(x)-{g}_1(x)^{n-1} {g}_2(x)]\\
				&=\frac{1}{p}[ x^{n}+ax^{n-1}+bx+c-x^{n-1}(x+{a})]\\
				&= \frac{1}{p}(bx+c)\\
				&= \frac{bx}{p}+\frac{c}{p}.
			\end{align*}
			Here, $\bar{g_1}(x)$ or $\bar{g_2}(x)$ divides $\bar{M}(x)$ if and only if either $p^{2}|c$ or $ \bar 
			{ \frac{c}{p}} = \bar{a}\bar {\frac{b}{p}},$ where $\bar {a},$ $\bar {\frac{b}{p}},$ $\bar {\frac{c}{p}}$ $\in \F_{p}.$ This completes the proof. 
		\end{proof}
		The following lemma tells about some special primes $p$ such that $p\nmid D.$
		\begin{lemma}\label{L7}
			Let $\theta$ be an algebraic integer and $$f(x)= x^{n}+ax^{n-1}+bx+c$$ be the minimal polynomial of $\theta$ over $\mathbb{Q}$ with $n^{2}=ak,$ $k \in \mathbb{N},$ $n>4,$ and $abc\neq 0.$ Let $K=\mathbb{Q}(\theta)$ and $\mathcal{O}_{K}$ be the ring of algebraic integers of $K.$ If $p$ is an odd prime with $p|a$ and $p\nmid b,$ then $p\nmid [\mathcal{O}_{K}:\mathbb{Z}[\theta]].$ 
			\begin{proof}
				Let $p|a$ and $p\nmid {b}.$ By using $n^{2}=ak$ in the value of $D$ (equation \ref{E1}), we obtain
				\begin{align*} 
					D&=(-1)^{\frac{(n+2)(n-1)}{2}}\bigg[(n-1)^{n-1}a^{n}c^{n-2}+k(n-1)^{n-1}b^{n-1}c+(n-1)^{n-3}b^{n-2}(kc-b)^{2} \nonumber\\
					&-(n-1)^{n-3}b^{n-2}(kc-b)((n-2)bn+ck)\nonumber\\
					&-(1+(-1)^{n})\frac{a(kc-b)[((n-2)ab+cn)^{2}-4abc(n-1)^{2}]^\frac{n-2}{2}}{2^{n-2}}\nonumber\\
					&+\sum_{i=0}^{\lfloor \frac{n-3}{2} \rfloor}\bigg[\frac{2n(n-1)^{2}abc[(2-n)ab-cn]^{n-3-2i}[((n-2)ab+cn)^{2}-4abc(n-1)^{2}]^{i}\binom{n-3}{2i}}{2^{n-3}}\nonumber\\
					&+\frac{a(ck-b)[(2-n)ab-cn]^{n-2-2i}[((n-2)ab+cn)^{2}-4abc(n-1)^{2}]^{i}\binom {n-3}{2i}(n-2)}{2^{n-3}(n-2-2i)}\bigg]\bigg]. 
				\end{align*}
				It is easy to verify that $\binom {n-3}{2i}\frac {(n-2)}{(n-2-2i)}$ is an integer and $p$ is an odd prime, therefore, $$ D\equiv 0 \pmod{p}$$ 
				if and only if
				$$ \bigg (k(n-1)^{n-1}b^{n-1}c+(n-1)^{n-3}b^{n-2}(kc-b)^{2}
				-(n-1)^{n-3}b^{n-2}(kc-b)((n-2)bn+ck)\bigg)\equiv 0 \pmod{p}.$$ 
				Since, $p|a$ implying $p|n,$ so $p\nmid (n-1)$ and also $p\nmid {b},$ hence we get
				$$\bigg (kbc+(kc-b)^{2} -ck(kc-b)\bigg)\equiv 0 \pmod{p},$$ 
				which is further implies that
				$$ b^{2}\equiv 0 \pmod{p}.$$ 
				This is a contradiction as $p\nmid b.$ Consequently, $p \nmid D.$ Finally, from the formula $D= [\mathcal{O}_{K}:\mathbb{Z}[\theta]]^{2}D_{K},$ 
				we conclude that $p \nmid [\mathcal{O}_{K}:\mathbb{Z}[\theta]].$ This completes the proof.
			\end{proof}
		\end{lemma}
		\section{\bf Proofs of the main theorems}
		\begin{proof}[\bf{Proof of Theorem \ref{th1}}]
			We prove each part of the theorem separately.~Consider the $\textbf{first part},$ when $p|a,$ $p|b,$ and $p|c,$ where $p$ is any prime. If we consider $a_{i}=0,$ for all $i= 2,\ldots,(n-2),$ $a_{0}=c,$ $a_{1}=b,$ and $a_{n-1}=a$ in Lemma (\ref{L5}), then the first part of the theorem holds trivially.
			
			Now, consider the \textbf{second part} when $p|a,$ $p|b$ and $p\nmid c.$ Since $p$ is a prime number and $n^{2}=ak$ therefore, $p|n.$ Let $n=p^{r}m,$ $r \in \N ,$ and $p\nmid m.$ Then, $$f(x)=x^{n}+ax^{n-1}+bx+c\equiv x^{n}+c \pmod{p},$$ i.e. $$\bar {f}(x)=x^{n}+\bar {c}=x^{p^{r}m}+\bar {c}\in \F_{p}[x] .$$ 
			Since $p\nmid c$ implies that $\gcd (p, c)=1$ and using Fermat's little theorem, we have $c^{p^{r}} \equiv c \pmod{p}.$ Now, using binomial theorem, we get $ {f}(x)\equiv x^{p^{r}m}+{c} \equiv (x^{m}+c)^{p^{r}} \pmod{p}.$ Let $\displaystyle\prod_{j=1}^{s} \bar g_{j}(x)$ be the factorization of $x^{m}+\bar {c}$ over the field $\F_{p}.$ We can write $$x^{m}+c= \displaystyle\prod_{j=1}^{s} g_{j}(x)+pU(x),$$ for some $U(x)\in \mathbb{Z}[x]$ and $g_{j}(x)\in \mathbb{Z}[x]$ are monic lifts corresponding to the distinct monic irreducible polynomial factors $\bar g_{j}(x).$ Now,
			\begin{align*}
				f(x)&=x^{n}+ax^{n-1}+bx+c\\
				&=t(x^{p^{r}})+ax^{n-1}+bx,
			\end{align*} 
			where $t(x)=x^{m}+c.$ By using Lemma (\ref{L3}), we have 
			\begin{align}
				\label{eq2}
				f(x)=\bigg(\displaystyle\prod_{j=1}^{s} g_{j}(x)+pU(x)\bigg)^{p^{r}}+p\bigg(\displaystyle\prod_{j=1}^{s} g_{j}(x)+pU(x) \bigg)
				V(x)+(c+(-c)^{p^{r}})+ax^{n-1}+bx,
			\end{align}
			for some $V(x)\in  \mathbb{Z}[x]$.
			Now $\bar {f}(x)=x^{n}+\bar {c}=(x^{m}+\bar {c})^{p^{r}}=\bigg(\displaystyle\prod_{j=1}^{s} \bar g_{j}(x)\bigg)^{p^{r}} .$ On substituting the value of $f(x)$ from (\ref{eq2}) in $M(x)$(\ref {E1}) and using Lemma (\ref{L1}) along with binomial theorem, we obtain $$\bar M(x)= \bigg(\displaystyle\prod_{j=1}^{s} \bar g_{j}(x)\bigg)\bar {V}(x)+\bar {b_1}x+\bar {c_1}, $$ where $b=pb_{1}$ and $c_{1}=\frac{(c+(-c)^{p^{r}})}{p}.$ From Theorem (\ref{T4}), $p\nmid [\mathcal{O}_{K}:\mathbb{Z}[\theta]]$ if and only if $\bar M(x)$ is co-prime to $\displaystyle\prod_{j=1}^{s} \bar g_{j}(x)$ or co-prime to  $\bigg(\displaystyle\prod_{j=1}^{s} \bar g_{j}(x)\bigg)^{p^{r}}=x^{n}+\bar {c}$ which is further implies that $\bar {b_1}x+\bar {c_1}$ is co-prime to $x^{n}+\bar {c}.$ Let $\xi$ be the common zero of $\bar {b_1}x+\bar {c_1}$ and $x^{n}+\bar {c}$ implies that $\bar {b_1}\xi+\bar {c_1}=\bar 0$ and $\xi^{n}+\bar {c}=\bar 0$ which is possible only if either $p|b_{1}$ and $p| c_{1}$ or if $p\nmid b_{1}$ and $p\nmid c_{1},$ then $p|[c(b_{1})^{n}+(-c_{1})^{n}].$ Conversely, $\bar {b_1}x+\bar {c_1}$ and $x^{n}+\bar {c}$ are co-prime if and only if either $p|b_{1}$ and $p\nmid c_{1}$ or $p\nmid b_{1}[c(b_{1})^{n}+(-c_{1})^{n}].$ This completes the proof of the second part.
			
			For the $\textbf{third part},$ by using Lemma (\ref {L6}), $p$ divides $[\mathcal{O}_{K}:\mathbb{Z}[\theta]]$ if and only if either $p^{2}|c$ or $ \bar 
			{ \frac{c}{p}}= \bar{a}\bar {\frac{b}{p}},$ where $\bar {a},$ $\bar {\frac{b}{p}},$ $\bar {\frac{c}{p}}$ $\in \F_{p}.$ Contrapositively,  $p$ does not divides $[\mathcal{O}_{K}:\mathbb{Z}[\theta]]$ if and only if $p^{2}\nmid {c}$ and $ \bar 
			{ \frac{c}{p}} \neq \bar{a}\bar {\frac{b}{p}},$ where $\bar {a},$ $\bar {\frac{b}{p}},$ $\bar {\frac{c}{p}}$ $\in \F_{p}.$
			
			\indent	Now, we discuss the $\textbf{fourth part}$ when $p\nmid ab$ and $p|c.$
			Firstly, we verify whether the polynomial $\bar f(x)=x^{n}+\bar {a}x^{n-1}+\bar {b}x$ has multiple zeros or not in the algebraic closure of the field $\F_{p}.$ As we know that a polynomial is inseparable if its derivative vanish at some root of it. But $$\bar f'(x)=\bar {n}x^{n-1}+\bar a ~\overline {(n-1)}x^{n-2}+\bar {b}$$ does not vanish at $x=\bar 0$ which implies that it is not a repeated zero of $\bar f(x).$ Therefore, $\bar f(x)$ is inseparable if and only if $x^{n-1}+\bar {a}x^{n-2}+\bar {b}=\bar h(x)~\text{(say)}$ is inseparable. Here, $$x^{n-3}[\overline {(n-1)}x+\bar {a}\overline {(n-2)}]=\bar h'(x).$$ If $p|(n-1),$ then it is easy to observe that $\bar h(x)$ is separable.
			Now, let $p\nmid (n-1)$ and $\alpha(\neq \bar 0)$ be the repeated zero of $\bar f(x)$ in the algebraic closure of $\F_{p}.$ Also, $\alpha $ is a repeated zero of $\bar f(x)$ if and only if $\bar h(\alpha)=\bar h'(\alpha)=\bar 0,$ i.e. 
			\begin{align}
				\label{eq3}
				\alpha^{n-1}+\bar {a}\alpha^{n-2}+\bar {b}=\bar 0
			\end{align}
			and
			\begin{align}
				\label{eq5}
				\overline{(n-1)}\alpha+\bar a ~\overline{(n-2)}=\bar 0.
			\end{align}
			Now, from equation (\ref{eq5}), we get $\alpha=-\overline{(n-1)}^{-1}\bar a~ \overline{(n-2)}.$ Using the value of $\alpha$ in equation (\ref{eq3}), we have $$(-\overline{(n-1)}^{-1}\bar a~ \overline{(n-2)})^{n-1}+\bar {a}(-\overline{(n-1)}^{-1}\bar a~ \overline{(n-2)})^{n-2}+\bar {b}=\bar 0$$ implies that $$ a(-a(n-2))^{n-2}+b(n-1)^{n-1}\equiv 0\pmod{p}.$$
			Hence, $\bar f(x)$ is separable if 
			\begin{align}
				\label{eq*}
				a(-a(n-2))^{n-2}+b(n-1)^{n-1}\not\equiv 0\pmod{p}.
			\end{align}
			Here, if $p=2,$ then (\ref {eq*}) holds trivially.\\ 
			Let $p$ be an odd prime and $ a(-a(n-2))^{n-2}+b(n-1)^{n-1}\equiv 0\pmod{p}.$ Then, from equation (\ref {eq5}), $\bar h(x)$ has only one repeated zero $\alpha=\bar n_{1} ~~\text{(say)}.$ Hence, $p\nmid [\mathcal{O}_{K}:\mathbb{Z}[\theta]]$ if and only if $(x-\bar n_{1})\nmid \bar M(x),$ thanks to the Dedekind criterion (\ref {T4}). This completes the proof of the fourth part.
			
			\indent	Consider the $\textbf{fifth part}$ when $p\nmid ac$ and $p|b.$ We prove this part is similar to the fourth part. 
			According to the given conditions, we have $\bar f(x)=x^{n}+\bar {a}x^{n-1}+\bar {c}$ and $$\bar f'(x)=\bar {n}x^{n-1}+\bar a ~\overline {(n-1)}x^{n-2}=x^{n-2}(\bar {n}x+\bar a ~\overline {(n-1)}).$$ From this, we see that $x=\bar 0 $ is not a zero of $\bar f(x),$ so it is not a repeated zero.
			Let $\gamma\neq \bar 0$ be a repeated zero of $\bar f(x)$ in the algebraic closure of $\F_{p}.$ If $p|n,$ then $\bar f'(\gamma)=\bar 0$ implies that $\gamma=\bar 0,$ which is a contradiction. Thus, $\bar f(x)$ is separable. Let $p\nmid n.$ Now $ \bar f'(\gamma)=\bar 0$ results in two possibilities, either $\gamma= \bar 0$ or $\bar {n}\gamma+\bar a ~\overline {(n-1)}=\bar 0$. However, $\gamma\neq \bar 0$ which implies that
			\begin{align}\label{eq100}
				\bar {n}\gamma+\bar a ~\overline {(n-1)}=\bar 0 ~~\text {or}~~\gamma=-(\bar n)^{-1}\bar a ~\overline {(n-1)}.
			\end{align}
			Putting the value of $\gamma$ in $\bar f(\gamma),$ we get $$\bar f(\gamma)=(-(\bar n)^{-1}\bar a ~\overline {(n-1)})^{n}+\bar a(-(\bar n)^{-1}\bar a ~\overline {(n-1)})^{n-1}+\bar c.$$ Here, $\gamma$ is a zero of $\bar f(x)$ if and only if $$\bar f(\gamma)=(-(\bar n)^{-1}\bar a ~\overline {(n-1)})^{n}+\bar a(-(\bar n)^{-1}\bar a ~\overline {(n-1)})^{n-1}+\bar c=\bar 0$$ which is further equivalent to $$(-\bar a ~\overline {(n-1)})^{n}+\bar a \bar n(-\bar a ~\overline {(n-1)})^{n-1}+\bar c\bar {n}^{n}=\bar 0.$$ Thus, $\bar f(x)$ is separable if
			\begin{align}\label{eq**}
				a(- a  {(n-1)})^{n-1}+ c{n}^{n}\not\equiv 0\pmod{p}.
			\end{align}
			Also, if $p=2,$ then (\ref {eq**}) holds trivially.\\
			Let $p$ be an odd prime and $ a(- a  {(n-1)})^{n-1}+ c{n}^{n}\equiv 0\pmod{p}.$ Then, from equation (\ref {eq100}), $\bar f(x)$ has only one repeated zero $\gamma=\bar n_{2} ~~\text{(say)}.$ Thus, $p\nmid [\mathcal{O}_{K}:\mathbb{Z}[\theta]]$ if and only if $(x-\bar n_{2})\nmid \bar M(x),$ thanks to the Dedekind criterion (\ref {T4}). This completes the proof of the fifth part.
			
			\indent	Consider the $\textbf{final part}$ when $p\nmid abc.$ We divide this part into two cases according to $p$ is an odd or even prime.
			
			\textbf{Case 6.1:} Let $p$ be an odd prime. Now, there are two possibilities that $\bar f(x)$ is separable or not. Assume $\zeta$ is a repeated zero of $\bar f(x).$ Then, $\bar f(\zeta)=\bar f'(\zeta)=\bar 0,$ where
			\begin{align}
				\label{eq17}
				\bar f(\zeta)=\zeta^{n}+\bar {a}\zeta^{n-1}+\bar b \zeta+\bar {c}=\bar 0
			\end{align}
			and
			\begin{align}
				\label{eq18}
				\bar f'(\zeta)=\bar {n}\zeta^{n-1}+\bar a ~\overline {(n-1)}\zeta^{n-2}+\bar {b}=\bar 0.
			\end{align}
			Now from equation (\ref {eq18}), we have the following subcases:
			
			\textbf{Subcase 6.1.1:}
			If $p|(n-1),$ then there exist two positive integers $r_0$ and $m_0$ such that $n-1=p^{r_0}m_0$ and $p\nmid m_0.$ Now, from equation (\ref {eq18}), we have  $\zeta^{n-1}=-\bar b.$ By using the value of $\zeta^{n-1}$ in equation (\ref {eq17}), we get $\zeta(-\bar b)-\bar a\bar b+\bar b\zeta+\bar c=\bar 0$ or $\bar a\bar b=\bar c.$ 
			
			Thus, $\bar f(x)$ is separable if $ ab\not \equiv c\pmod{p}.$
			Suppose $\bar a\bar b=\bar c,$ then by using binomial theorem, we have  
			$	\bar f(x)=x^{n}+\bar {a}x^{n-1}+\bar b x+\bar {c}=(x+\bar a)(x^{m_0}+\bar b)^{p^{r_0}}.$
			Let $\displaystyle\prod_{j=1}^{q} \bar F_{j}(x)$ be the factorization of $x^{m_0}+\bar {b}$ over the field $\F_{p},$ where $q\in \N.$ We can write 
			\begin{align}
				\label {eq 101}
				x^{m_0}+b= \displaystyle\prod_{j=1}^{q} F_{j}(x)+pU_{1}(x),
			\end{align}
			for some $U_{1}(x)\in \mathbb{Z}[x]$ and $F_{j}(x)\in \mathbb{Z}[x]$ are monic lifts corresponding to the distinct monic irreducible polynomial factors $\bar F_{j}(x).$ 
			Consider,
			\begin{align}
				f(x)&=x^{n}+ax^{n-1}+bx+c\\
				\label {eq 102}
				&= (x+a)(x^{m_0})^{p^{r_0}}+bx+c.
			\end{align} 
			From equation (\ref {eq 101}), on substituting the value of $x^{m_{0}}$ in (\ref {eq 102}), we have 
			\begin{align}
				\label{eq201}
				f(x)=(x+a)\bigg(\displaystyle\prod_{j=1}^{q} F_{j}(x)+pU_1(x)-b\bigg)^{p^{r_0}}+bx+c.
			\end{align}
			From (\ref{eq201}), putting the value of $f(x)$ in $M(x)$(\ref {E1}) and applying the binomial theorem, we get $$\bar M(x)= (x+\bar a)\bigg(\displaystyle\prod_{j=1}^{q} \bar F_{j}(x)\bigg)\bar {V_1}(x)+\bar {v_1}x+\bar {v_0}, $$ where $v_{1}=\frac{(b+(-b)^{p^{r_0}})}{p}$ and $v_{0}=\frac{(c+a(-b)^{p^{r}})}{p}.$ By using Theorem (\ref{T4}), $p\nmid [\mathcal{O}_{K}:\mathbb{Z}[\theta]]$ if and only if $\bar M(x)$ is co-prime to $(x+\bar a)\displaystyle\prod_{j=1}^{q} \bar F_{j}(x)$ or co-prime to  $(x+\bar a)\bigg(\displaystyle\prod_{j=1}^{q} \bar F_{j}(x)\bigg)^{p^{r_0}}=x^{n}+\bar {a}x^{n-1}+\bar b x+\bar {c}$ which is futher implies that $\bar {v_1}x+\bar {v_0}$ is co-prime to $x^{n}+\bar {a}x^{n-1}+\bar b x+\bar {c}.$ Let $\beta$ be the common zero of $\bar {v_1}x+\bar {v_0}$ and $x^{n}+\bar {a}x^{n-1}+\bar b x+\bar {c}$ implies that $\bar {v_1}\beta+\bar {v_0}=\bar 0$ and $\beta^{n}+\bar {a}\beta^{n-1}+\bar b \beta+\bar {c}=\bar 0$ which is possible only if either $p|v_{1}$ and $p| v_{0}$ or if $p\nmid v_{1}$ and $p\nmid v_{0},$ then $p|[(- v_{0})^{n}+ {a} v_{1}(- v_{0})^{n-1}- b ( v_{1})^{n-1} v_{0}+ {c}( v_{1})^{n}].$ Conversely, $\bar {v_1}x+\bar {v_0}$ and $x^{n}+\bar {a}x^{n-1}+\bar b x+\bar {c}$ are co-prime if and only if either $p|v_{1}$ and $p\nmid v_{0}$ or $p\nmid v_{1}[(- v_{0})^{n}+ {a} v_{1}(- v_{0})^{n-1}- b ( v_{1})^{n-1} v_{0}+ {c}( v_{1})^{n}].$
			
			\textbf{Subcase 6.1.2:}
			If $p|n,$ then $p\nmid (n-1).$ Now from equation (\ref {eq18}), we have 
			\begin{align}
				\label{eq19}
				\zeta^{n-2}=(\bar a)^{-1}\bar {b}.
			\end{align}
			Using the value of $\zeta^{n-2}$ in equation (\ref {eq17}), we get $\zeta^{2}+2\bar a\zeta+\bar a\bar c(\bar b)^{-1}=\bar 0.$ On solving this quadratic equation, we have $$\zeta=\frac {-2\bar a\pm[(2\bar a)^{2}-4\bar a\bar c(\bar b)^{-1}]^{\frac {1}{2}}}{2}=-\bar a\pm[(\bar a)^{2}-\bar a\bar c(\bar b)^{-1}]^{\frac {1}{2}}.$$ Now, putting the value of $\zeta$ in equation (\ref {eq19}), we obtain $$\bar a(-\bar a\pm[(\bar a)^{2}-\bar a\bar c(\bar b)^{-1}]^{\frac {1}{2}})^{n-2}=\bar {b}.$$
			
			Thus, $\bar f(x)$ is separable if $ \bar a(-\bar a\pm[(\bar a)^{2}-\bar a\bar c(\bar b)^{-1}]^{\frac {1}{2}})^{n-2}\neq \bar {b}.$
			Now, if $\bar a(-\bar a\pm[(\bar a)^{2}-\bar a\bar c(\bar b)^{-1}]^{\frac {1}{2}})^{n-2}=\bar {b},$ then $\bar f(x)$ has only two possible zeros $\zeta_{1}=-\bar a+[(\bar a)^{2}-\bar a\bar c(\bar b)^{-1}]^{\frac {1}{2}}$ and $\zeta_{2}=-\bar a-[(\bar a)^{2}-\bar a\bar c(\bar b)^{-1}]^{\frac {1}{2}}.$  Using Dedekind criterion (\ref {T4}), $p\nmid [\mathcal{O}_{K}:\mathbb{Z}[\theta]]$ if and only if $(x-\zeta_{j})\nmid\bar M(x),$ where $j=1,~2.$ 
			
			\textbf{Subcase 6.1.3:}
			If $p\nmid n (n-1),$ then from equation (\ref {eq18}), we have $$\bar {n}\zeta^{n-1}+\bar a ~\overline {(n-1)}\zeta^{n-2}+\bar {b}=\bar 0$$ or
			\begin{align}
				\label{eq20}
				\zeta^{n-2}(\bar {n}\zeta+\bar a ~\overline {(n-1)})+\bar {b}=\bar 0.
			\end{align}
			Here, $\bar {n}\zeta+\bar a ~\overline {(n-1)}$ can not be zero because if it is zero then $\bar b=\bar 0,$ which is not possible (since $p\nmid b$). Therefore, the inverse of $\bar {n}\zeta+\bar a ~\overline {(n-1)}$ exists.
			Now, from equation (\ref {eq20}), we get $$\zeta^{n-2}=-(\bar {n}\zeta+\bar a ~\overline {(n-1)})^{-1}\bar {b}.$$ Using the value of $\zeta^{n-2}$ in equation (\ref {eq17}), we have
			$$\zeta^{2}(-(\bar {n}\zeta+\bar a ~\overline {(n-1)})^{-1}\bar {b})+\bar {a}\zeta(-(\bar {n}\zeta+\bar a ~\overline {(n-1)})^{-1}\bar {b})+\bar b \zeta+\bar {c}=\bar 0$$ or $$\zeta^{2}+\zeta(\bar b-\bar n\bar b)^{-1}(\bar a\bar b-\overline{(n-1)}\bar a \bar b-\bar c\bar n)-(\bar b-\bar n\bar b)^{-1}\overline{(n-1)}\bar a \bar c=\bar 0.$$ The zeros of above quadratic equation is given by 
			\begin{equation*}
				\begin{split}
					\zeta&=\frac{1}{2}\bigg(-(\bar b-\bar n\bar b)^{-1}(\bar a\bar b-\overline{(n-1)}\bar a \bar b-\bar c\bar n)\\&\pm[((\bar b-\bar n\bar b)^{-1}(\bar a\bar b-\overline{(n-1)}\bar a \bar b-\bar c\bar n))^{2}+4(\bar b-\bar n\bar b)^{-1}\overline{(n-1)}\bar a \bar c]^\frac{1}{2}\bigg)=\frac{l_{1}}{2}~(\text{say}).
				\end{split}
			\end{equation*}
			Now, substituting the value of $\zeta$ in equation (\ref {eq20}), we get $$\bigg (\frac{l_{1}}{2}\bigg)^{n-2}\bigg(\bar {n}\bigg(\frac{l_{1}}{2}\bigg)+\bar a ~\overline {(n-1)}\bigg)+\bar {b}=\bar 0$$ or $$(l_{1})^{n-2}(\bar {n}l_{1}+2\bar a ~\overline {(n-1)})+2^{n-1}\bar {b}=\bar 0.$$
			
			Thus, $\bar f(x)$ is separable if $(l_{1})^{n-2}(\bar {n}l_{1}+2\bar a ~\overline {(n-1)})+2^{n-1}\bar {b}\neq \bar 0.$
			Now, if $$(l_{1})^{n-2}(\bar {n}l_{1}+2\bar a ~\overline {(n-1)})+2^{n-1}\bar {b}=\bar 0,$$ then $\bar f(x)$ has only two possible zeros $\zeta_{3}=\frac{1}{2}(A+B)$ and $ \zeta_{4}=\frac{1}{2}(A-B),$ where $$A=-(\bar b-\bar n\bar b)^{-1}(\bar a\bar b-\overline{(n-1)}\bar a \bar b-\bar c\bar n)$$ and $$B=[((\bar b-\bar n\bar b)^{-1}(\bar a\bar b-\overline{(n-1)}\bar a \bar b-\bar c\bar n))^{2}+4(\bar b-\bar n\bar b)^{-1}\overline{(n-1)}\bar a \bar c]^\frac{1}{2}.$$ By using Dedekind criterion (\ref {T4}), $p\nmid [\mathcal{O}_{K}:\mathbb{Z}[\theta]]$ if and only if $(x-\zeta_{j})\nmid\bar M(x),$ where $j=3,~4.$ 
			
			\textbf{Case 6.2:} If $p=2,$ then $\bar f(x)=x^{n}+x^{n-1}+x+1=(x^{n-1}+1)(x+1),$ hence $\bar 1$ is a repeated zero of $\bar f(x).$ Now, we discuss about the common zeros of $\bar f(x)$ and $\bar M(x)$(\ref {E1}). Again, we divide this case into two subcases.
			
			\textbf{Case 6.2.1:}
			If $p=2$ and $n$ is an even integer, then we have $$\bar f(x)=x^{n}+x^{n-1}+ x+ {1}=(x+1)(x^{n-1}+1)=(x+1)\mu (x),$$ where $\mu (x)=x^{n-1}+1.$ Now, $\mu' (x)=\overline{(n-1)}x^{n-2}=x^{n-2}.$ From here, it is clear that $\mu (x)$ and $\mu'(x)$ have no common zeros in the algebraic closure of $\F_{p}.$ Therefore, $\mu(x)$ is separable over the field $\F_{p}.$
			Also, $(x+1)$ is a factor of $\mu (x).$ Therefore, using Dedekind criterion (\ref {T4}), $p\nmid[\mathcal{O}_{K}:\mathbb{Z}[\theta]]$ if and only if $(x+1)\nmid \bar M(x)$(\ref {E1}).  
			
			\textbf{Case 6.2.2:}
			If $p=2$ and $n$ is an odd number. Then, there exist a number $r_1\in \mathbb{N}$ such that $n-1=2^{r_1}k_{1} ~\text{(say)},$ $k_{1}$ is odd and $$\bar f(x)=x^{n}+x^{n-1}+ x+ {1}=(x+1)(x^{k_{1}}+1)^{2^{r_1}}.$$ By our assumption, $\displaystyle\prod_{i=1}^{u} \bar h_{i}(x)$ be the factorization of $x^{k_{1}}+ {1}$ over the field $\F_{p}.$ Thus, we have 
			\begin{equation}
				\label{eq2001}
				x^{k_{1}}+ {1}= \displaystyle\prod_{i=1}^{u} h_{i}(x)+2U_{2}(x), ~\text{for some} ~~U_{2}(x)\in \mathbb{Z}[x].
			\end{equation} 
			Further, we write
			\begin{align}
				\label{eq2002}
				f(x)&=x^{n}+ax^{n-1}+bx+c\\
				&=(x+a)(x^{k_{1}})^{2^{r_1}}+bx+c.
			\end{align} 
			From equation (\ref{eq2001}), using the value of $x^{k_{1}}$ in equation (\ref {eq2002}), we obtain
			\begin{align}
				\label{eq2000}
				f(x)=(x+a)\bigg[\bigg(\displaystyle\prod_{i=1}^{u} h_{i}(x)+2U_{2}(x)-1\bigg)^{2^{r_1}}\bigg]+bx+c.
			\end{align}
			
			On substituting the value of $f(x)$ from (\ref{eq2000}) in $M(x)(\ref {E1}),$ we have $$\bar M(x)= \bigg(\displaystyle\prod_{i=1}^{u} \bar h_{i}(x)\bigg)\bar {W}(x)+\overline {\frac{b+1}{2}}x+\overline  {\frac{a+c}{2}}, $$ where ${W}(x)$ contain other remaining terms. Now, let $\nu$ be the common zero of $\bar f(x)$ and $\bar M(x).$ Then, $\bar M(\nu)= \overline {\frac{b+1}{2}}\nu+\overline  {\frac{a+c}{2}}=\bar {0} $ implies that if $2|(\frac{b+1}{2}),$ then $2|(\frac{a+c}{2})$ and if $2\nmid (\frac{b+1}{2}),$ then $2\nmid (\frac{a+c}{2})$ and conversely. Otherwise, $\bar f(x)$ and $\bar M(x)$ have no common zeros. Thus, by using the Dedekind criterion(\ref {T4}), $2$ does not divide $[\mathcal{O}_{K}:\mathbb{Z}[\theta]]$ if and only if either $2|(\frac{b+1}{2})$ and $2\nmid (\frac{a+c}{2})$ or $2\nmid (\frac{b+1}{2})$ and $2| (\frac{a+c}{2})$
			Hence, taking the cases \textbf {6.1} and \textbf {6.2} together, we complete the proof of the final part.
			
			From the Lemma (\ref {L7}), there does not exist any odd prime $p$ which satisfies the given hypothesis along with the condition $p|a$ and $p\nmid b.$ 
			If $p=2$ and $2|a$ and $2\nmid b,$ then $\bar f(x)$ is separable. Therefore, using the Dedekind criterion(\ref {T4}), $2\nmid [\mathcal{O}_{K}:\mathbb{Z}[\theta]].$
			
			This completes the proof of the theorem.
			
		\end{proof}
		\begin{proof}[\bf{Proof of Corollary \ref{C1}}]
			The proof of the corollary follows from Theorem (\ref {th1}). Indeed, if each prime $p$ divides $D$ and satisfies one of the following conditions from (1) to (8) of Theorem (\ref {th1}), then $p\nmid[\mathcal{O}_{K}:\mathbb{Z}[\theta]].$ Therefore, by using the formula $$D= [\mathcal{O}_{K}:\mathbb{Z}[\theta]]^{2}D_{K},$$ we have $[\mathcal{O}_{K}:\mathbb{Z}[\theta]]=1$ implies that $\mathcal{O}_{K}=\mathbb{Z}[\theta].$ The Converse of the corollary holds trivially (Theorem \ref {th1}). This completes the proof. 
		\end{proof}
		
		\begin{proof}[\bf{Proof of Proposition \ref{P1}}]
			Let $p$ be an odd prime satisfying the conditions $p \nmid a,$ $p \nmid b,$ $p| c,$ $p|(n-1),$ and $abc\neq 0,$ $n^{2}=ak,$ $k \in \mathbb{N}.$ Let $ D\equiv 0 \pmod{p}.$ Now, using $n^{2}=ak$ in the value of $D$ as obtained in equation (\ref{E1}), we get
			\begin{align}\label{eq40} 
				D&=(-1)^{\frac{(n+2)(n-1)}{2}}\bigg[(n-1)^{n-1}a^{n}c^{n-2}+k(n-1)^{n-1}b^{n-1}c+(n-1)^{n-3}b^{n-2}(kc-b)^{2} \nonumber\\
				&-(n-1)^{n-3}b^{n-2}(kc-b)((n-2)bn+ck)\nonumber\\
				&-(1+(-1)^{n})\frac{a(kc-b)[((n-2)ab+cn)^{2}-4abc(n-1)^{2}]^\frac{n-2}{2}}{2^{n-2}}\nonumber\\
				&+\sum_{i=0}^{\lfloor \frac{n-3}{2} \rfloor}\bigg[\frac{2n(n-1)^{2}abc[(2-n)ab-cn]^{n-3-2i}[((n-2)ab+cn)^{2}-4abc(n-1)^{2}]^{i}\binom{n-3}{2i}}{2^{n-3}}\nonumber\\
				&+\frac{a(ck-b)[(2-n)ab-cn]^{n-2-2i}[((n-2)ab+cn)^{2}-4abc(n-1)^{2}]^{i}\binom {n-3}{2i}(n-2)}{2^{n-3}(n-2-2i)}\bigg]\bigg]. 
			\end{align}
			Also, we can verify that $\binom {n-3}{2i}\frac {(n-2)}{(n-2-2i)}$ is an integer and $p$ is an odd prime, therefore, $ D\equiv 0 \pmod{p}$ 
			leading to the following implication:  
			
			$$\hspace{-7cm}\bigg (-(1+(-1)^{n})\frac {(-ab)((n-2)ab)^{n-2}}{2^{n-2}}
			$$ 
			$$\hspace{2cm}+~(-1)^{n-2}\frac {(-ab)((n-2)ab)^{n-2}}{2^{n-3}}\sum_{i=0}^{\lfloor \frac{n-3}{2} \rfloor}\binom{n-3}{2i}\frac{n-2}{n-2-2i}\bigg)\equiv 0 \pmod{p}.
			$$ This further reduces to
			$$\frac{(-ab)((n-2)ab)^{n-2}}{2^{n-2}}\bigg(2(-1)^{n-2}\sum_{i=0}^{\lfloor \frac{n-3}{2} \rfloor}\binom{n-3}{2i}\frac{n-2}{n-2-2i}-(1+(-1)^{n})\bigg)\equiv 0 \pmod{p}.$$ But $p\nmid(-ab)$ since if $p|(-ab),$ then $p|a$ or $p|b$ which is not possible due to our assumption$ p \nmid a,$ $p \nmid b,$ and also $p\nmid ab(n-2)$ because if $p|ab(n-2),$ then $p|(n-2)$ and this is impossible (since $p|(n-1)$). Thus, we have $$\bigg(2(-1)^{n-2}\sum_{i=0}^{\lfloor \frac{n-3}{2} \rfloor}\binom{n-3}{2i}\frac{n-2}{n-2-2i}-(1+(-1)^{n})\bigg)\equiv 0 \pmod{p} $$ or $$
			\left\{\begin{array}{ll}
				\displaystyle\sum_{i=0}^{\lfloor \frac{n-3}{2} \rfloor}\binom{n-3}{2i}\left[\frac{n-2}{n-2-2i}\right] \equiv 1\pmod{p}, &  {\rm if }\; n \text{~is~even},\\
				\displaystyle\sum_{i=0}^{\lfloor \frac{n-3}{2} \rfloor}\binom{n-3}{2i}\left[\frac{n-2}{n-2-2i}\right]\equiv 0\pmod{p}, &  {\rm if }\; n \text{~is~odd}. 
			\end{array}\right.$$ 
			Again, from the part (4) of Theorem (\ref {th1}), we have $p \nmid [\mathcal{O}_{K}:\mathbb{Z}[\theta]].$ Now, from the well known formula $D= [\mathcal{O}_{K}:\mathbb{Z}[\theta]]^{2}D_{K},$ it is clear that $$p|D_{K}~~ \text {if and only if}~~	\left\{\begin{array}{ll}
				\displaystyle\sum_{i=0}^{\lfloor \frac{n-3}{2} \rfloor}\binom{n-3}{2i}\left[\frac{n-2}{n-2-2i}\right]\equiv 1\pmod{p}, &  {\rm if }\; n \text{~is~even},\\
				\displaystyle\sum_{i=0}^{\lfloor \frac{n-3}{2} \rfloor}\binom{n-3}{2i}\left[\frac{n-2}{n-2-2i}\right]\equiv 0\pmod{p}, &  {\rm if }\; n \text{~is~odd}. 
			\end{array}\right. $$ This completes the proof.
		\end{proof}
		\begin{proof}[\bf{Proof of proposition \ref{T2}}]
			Consider the first case, when $p\nmid a,$ $p\nmid c ,$ $p|b,$ and $p|n.$ Suppose $p|D.$ Then, from (\ref {eq40}), we observe that $ (n-1)^{n-1}a^{n}c^{n-2}\equiv 0 \pmod{p},$ which is a contradiction as $p\nmid a,$ $p\nmid c,$ and $p\nmid (n-1)$ (because $p|n$). Thus, $p\nmid D$ implies that $p\nmid [\mathcal{O}_{K}:\mathbb{Z}[\theta]].$\\Now, we prove  the final case, when $p\nmid a,$ $p\nmid b,$ $p|c,$ and $p|(n-2).$ Again, in a similar process, from (\ref{eq40}), if $p|D,$ then $ (n-1)^{n-3}b^{n-2}(kc-b)^{2}\equiv 0 \pmod{p}$ or $ b^{n}(n-1)^{n-3}\equiv 0 \pmod{p},$ which is again a contradiction as $p\nmid b$ and $p\nmid (n-1)$ (because $p|(n-2)$). Hence, $p\nmid {D}$ and this indicates that  $p\nmid [O_{K}:\mathbb{Z}[\theta]].$ This completes the proof.
		\end{proof}
		
		\section{\bf Examples}	 
		In this segment, we present a few examples that demonstrate the outcomes we have obtained. In the following examples except (\ref{ex}), $K=\mathbb{Q}(\theta)$ be an algebraic number field related to algebraic integer $\theta$ with minimal polynomial $f(x)$ and $\mathcal{O}_{K}$ denotes the ring of algebraic integers of the number field $K.$ 
		\begin{example}
			Let $f(x)= x^{6}+4x^{5}+x+3$ be the minimal polynomial of $\theta$ over $\mathbb{Q}.$ Here, $|D|=3^{2}.7561.15269$ and by part (4) $3| [\mathcal{O}_{K}:\mathbb{Z}[\theta]]$ because $4^{9}+5^{5}\equiv 0 \pmod{3},$ $3\nmid 5,$ and $(x+2)|\bar M(x),$ where $M(x)=x^4+2x^2+2x+1.$ As we know that the discriminant $D$ is expressed as $D= [\mathcal{O}_{K}:\mathbb{Z}[\theta]]^{2}D_{K},$ therefore, $[\mathcal{O}_{K}:\mathbb{Z}[\theta]]=3$ which implies that the field  $K$ is not monogenic with respect to $\theta$.
		\end{example}	
		\begin{example}
			Let $f(x)= x^{5}+x^{4}+3x+6$ be the minimal polynomial of $\theta$ over $\mathbb{Q}.$ Here, the question arises whether the number field $K$ is monogenic or not. Then, the answer will be YES, because $|D|=3^{3}.5.18691$ and by part (3) of Theorem (\ref{th1}) $3\nmid [\mathcal{O}_{K}:\mathbb{Z}[\theta]].$ Also, the discriminant $D$ is expressed as $D= [\mathcal{O}_{K}:\mathbb{Z}[\theta]]^{2}D_{K},$ therefore, $[\mathcal{O}_{K}:\mathbb{Z}[\theta]]=1$ implying that the field $K$ is monogenic.
		\end{example}	
		\begin{example}
			Let $f(x)= x^{6}+9x^{5}+3x+18$ be the minimal polynomial of $\theta$ over $\mathbb{Q}.$ Here, $3|a,~ 3|b,~ 3|c, ~~\text{and}~ ~3^{2}|c,~\text{ where}~~ a=9,~ b=3, ~\text{and}~~c=18.$ Then, by part (1) of Theorem (\ref{th1}), $3|[\mathcal{O}_{K}:\mathbb{Z}[\theta]]$ which implies that the field $K$ is not monogenic with respect to $\theta$.
		\end{example}	
		
		\begin{example}\label{ex}
			Let $f(x)= x^{30030}+44100x^{30029}+143x+7507$ be the polynomial over $\mathbb{Q}.$ Then, using Theorem (\ref{T2}) and Lemma (\ref {L7}), the discriminant $D$ of the given polynomial $f(x)$ does not belong to the ideals $$2\mathbb{Z}, ~3\mathbb{Z},~ 5\mathbb{Z},~ 7\mathbb{Z}, ~11\mathbb{Z},  ~13\mathbb{Z}, ~\text{and} ~7507\mathbb{Z}$$ of ring $\mathbb{Z}.$	
		\end{example}
		
		\section{Data Availablity} 	
		The authors confirm that their manuscript has no associated data.
		\section{Competing Interests}
		The authors confirm that they have no competing interest. 
		
		\section{Acknowledgement} Karishan Kumar extends his gratitude to the CSIR fellowship for partial support under the file no: 
		09/1005(16567)/2023-EMR-I.

	\end{document}